\documentclass[12pt,a4paper,reqno]{amsart}
\usepackage{xcolor}
\usepackage{amssymb}
\usepackage{amsmath,amsthm}
\usepackage{txfonts}
\usepackage{mathrsfs}

\usepackage{latexsym,euscript,dsfont}
\usepackage{bbm}
\usepackage{colortbl}
\usepackage{cite}
\usepackage{amsmath,amsthm,amsfonts,amssymb}
\usepackage{latexsym}
\usepackage{enumerate}
\usepackage{mathrsfs}
\def\bc{\begin{center}}
\def\ec{\end{center}}
\def\be{\begin{equation}}
\def\ee{\end{equation}}
\def\F{\mathbb F}
\def\N{\mathbb N}

\newcommand\hdim{\dim_{\mathrm H}}

\def\A{\textsf{A}}

\newtheorem{lem}{Lemma}[section]

\newtheorem{dfn}[lem]{Definition}
\newtheorem{pro}[lem]{Proposition}
\newtheorem{thm}[lem]{Theorem}

\newtheorem{rem}{Remark}
\numberwithin{equation}{section}

\usepackage{colortbl}
\begin{document}
\title[Divergence points in $\beta$-dynamical systems]{Multifractal analysis of the divergence points of Birkhoff averages in $\beta$-dynamical systems}

\author[Y. H. Chen]{Yuanhong Chen}
\address{School of Mathematics and Statistics\\ Huazhong University of Science and Technology\\
430074 Wuhan, Hubei, P.R. China}
\email{yhchen317@163.com}

\author[Z. L. Zhang]{Zhenliang Zhang*}%{\thanks Corresponding author}
\address{School of Mathematics and Statistics\\ HuaZhong Universitry of Science and Technology\\
430074 Wuhan, Hubei, P. R. China}
\email{zhliang\_zhang@163.com}

\author[X. J. Zhao]{Xiaojun Zhao}
\address{School of Economics\\Peking University\\
100871  Peking, P.R. China}
\email{zhaoxiaojun@pku.edu.cn}

\date{}

\begin{abstract}
This paper is aimed at a detailed study of the multifractal analysis of the so-called divergence points
 in the system of $\beta$-expansions.
More precisely,
let $([0,1),T_{\beta})$ be the $\beta$-dynamical system for a general $\beta>1$
and $\psi:[0,1]\mapsto\mathbb{R}$ be a continuous function.
Denote by $\textsf{A}(\psi,x)$ all the accumulation points of
$\Big\{\frac{1}{n}\sum_{j=0}^{n-1}\psi(T^jx): n\ge 1\Big\}$.
The Hausdorff dimensions of the sets
$$\Big\{x:\textsf{A}(\psi,x)\supset[a,b]\Big\},\ \ \Big\{x:\textsf{A}(\psi,x)=[a,b]\Big\}, \
\Big\{x:\textsf{A}(\psi,x)\subset[a,b]\Big\}$$
i.e.,
the points for which the Birkhoff averages of $\psi$ do not exist but behave
in a certain prescribed way,
are determined completely for any continuous function $\psi$.
\end{abstract}


\maketitle

\section{Introduction}
Let $(X,T)$ be a dynamical system.
Given an integrable function $\psi$,
call $x\in X$ a $\psi$-divergence point,
or simply divergence point, if the limit of the Birkhoff averages
 \begin{equation}
 \label{f1}
 \lim_{n\rightarrow\infty}\frac{1}{n}\sum_{j=0}^{n-1}\psi(T^{j}x):=\lim_{n\to\infty}\frac{1}{n}S_n\psi(x)
 \end{equation}
 does not exist.
 In the sense of Birkhoff's ergodic theorem,
 the divergence points are not detectable for any invariant probability measure.
 However,
 it is known that the divergence points can be large from the point of view of dimension theory,
 once $\psi$ is not cohomologous to a constant (see for examples \cite{Ba,FFW,Th}).
 Moreover, as Ruelle said,
 points with converging Birkhoff averages can only see average behavior,
 while the divergence points would reflect a finer structure of the system \cite{Ru}.
 This leads to a rich study on the structure of these points.
 Barreira and Schmeling initiated the study about the size of the divergence points in Markov systems \cite{Ba},
 which was also extended to systems of conformal repeller,
 conformal horseshoes, $\beta$-expansions,
 see \cite{BS,Th,Chen,T} and references therein.

To have a better understanding of the divergence points
and to provide extremely precise quantitative information
about the distribution of the individual divergence points,
Olsen \cite{O3} initiated a detailed study of the fractal structure of those points.
More precisely, the multifractal decomposition sets were considered,
where the Birkhoff averages diverge in a prescribed way:
 \begin{enumerate}[(1)]
  \item $E_{\supset[a,b]}=\Big\{x\in X:\textsf{A}(\psi,x)\supset[a,b]\Big\};$
  \item $E_{=[a,b]}=\Big\{x\in X:\textsf{A}(\psi,x)=[a,b]\Big\};$
  \item $E_{\subset[a,b]}=\Big\{x\in X:\textsf{A}(\psi,x)\subset[a,b]\Big\},$
\end{enumerate}
where $\textsf{A}(\psi,x)$ denotes the set of accumulation points of $\Big\{\frac{1}{n}S_n\psi(x): n\ge 1\Big\}$. One is referred to a series of work of Olsen \cite{O1,O2,O3,O4},  Olsen \& Winter \cite{O5,O6}, Olsen, Baek \& Snigireva \cite{BO} and references therein.

 It should be pointed that most of them studied the dimensions of the sets  defined above in the systems with Markov properties. In this paper, we focus on the $\beta$-expansions  which is a non-Markov property for a general $\beta>1$. This non-Markov property always plays a main barrier in studying the metrical properties of $\beta$-expansion. In order to understand better the non-Markov property and find ways to conquer difficulties caused by it, we aim at giving a detailed study of the multifractal analysis of the divergence points in $\beta$-expansions, by following the setting of Olsen given above.

Let us first fix some notation. Let $\beta>1$ and $T_{\beta}$ be the $\beta$-transformation given by
$$ T_{\beta}(x)=\beta x-\lfloor\beta x\rfloor,\ \ x\in [0,1)$$
where $\lfloor\cdot\rfloor$ denotes the integer part of a real number.
It is well known that $T_{\beta}$ is invariant and ergodic with respect to the Parry
measure \cite{P} given by $d{\nu}=\sum_{n:x\leq T_{\beta}^n1}\beta^{-n}dx.$ Then an application of the Birkhoff's ergodic Theorem yields that for any integrable function $\psi,$
$$\lim_{n\rightarrow\infty}\frac{1}{n}\sum_{j=0}^{n-1}\psi(T_{\beta}^j(x))=\int\psi d{\nu},\ \ \nu\text{-almost surely}.$$

 Recall that
 $\A(\psi, x)$ is the set of accumulation points of the sequence \break
  $\Big\{\frac{1}{n}S_n\psi(x)\Big\}_{n\geq1}$.
  Put $\mathfrak{L}_{\psi}=\bigcup_{x\in[0,1)}\A(\psi, x).$ When $\psi$ is continuous, $\mathfrak{L}_{\psi}$ is a closed interval and is given as
$$\mathfrak{L}_{\psi}=\Bigg\{\int\psi d{\mu}: \mu\ \text{is}\ T_{\beta}\text{-invariant}\Bigg\}.$$

The classical multifractal analysis of Birkhoff averages in $\beta$-dynamical system was extensively studied by Fan, Feng and Wu \cite{FFW} when $\beta$ is a Parry number, and by Pfister and Sullivan \cite{PS} for general $\beta>1.$ For any $\alpha\in\mathfrak{L}_{\psi},$ the dimension of the set
$$E_{\alpha}=\Bigg\{x\in[0,1):\lim_{n\rightarrow\infty}\frac{1}{n}S_n\psi(x)=\alpha\Bigg\}$$
is given by a variational principle:
\begin{equation}\label{f2}\hdim E_{\alpha}=\sup\Bigg\{\frac{h_{\mu}}{\log\beta}:\mu\in\mathfrak{M}(T_{\beta})\ \text{and}\ \int\psi d\mu=\alpha\Bigg\},\end{equation}
where $\mathfrak{M}(T_{\beta})$ is the collection of all $T_{\beta}$-invariant probability measures and $h_{\mu}$ is the measure theoretic entropy of $\mu$ (see \cite{W}, Chapter 4 for a definition of entropy). Denote $h_{\beta}(\alpha)$ the dimension of $E_{\alpha}$ for short.

%Now we consider the Hausdorff dimensions of the three multifractal decomposition sets.
In this paper, we focus on the dimensions of the sets $E_{\supset[a,b]}$, $E_{=[a,b]}$
  and $E_{\subset[a,b]}$ defined before. Let $\psi$ be a continuous function and $a$ and $b$ be two real numbers with $a<b$. Then $[a,b]\bigcap\mathfrak{L}_{\psi}$ is a closed interval, and we call it non-degenerate if it is nonempty. It is trivial to see that if $[a,b]\cap \mathfrak{L}_{\psi}$ is empty, $$
E_{\supset[a,b]}=[0,1), \ E_{=[a,b]}=\emptyset, \ \ E_{\subset[a,b]}=\emptyset.
$$ So we exclude this trivial case.
Our main result is stated as follows.
\begin{thm}
\label{t1}
Let $\beta>1$ and $\psi$ be a continuous function.
Let $a,b$ be two real numbers with $[a,b]\bigcap\mathfrak{L}_{\psi}$ non-degenerate.
Then
\begin{enumerate}[(1)]
  \item $\hdim E_{\supset[a,b]}=\hdim E_{=[a,b]}=\inf\Big\{h_{\beta}(\alpha):\alpha\in [a,b]\cap\mathfrak{L}_{\psi}\Big\},$
 % \item =\inf\Big\{h_{\beta}(\alpha):\alpha\in [a,b]\cap\mathfrak{L}_{\psi}\Big\};$
  \item $\hdim E_{\subset[a,b]}=\sup\Big\{h_{\beta}(\alpha):\alpha\in [a,b]\cap\mathfrak{L}_{\psi}\Big\},$
\end{enumerate} where $h_{\beta}(\alpha)$ is the dimension of $E_{\alpha}$ given in (\ref{f2}).
\end{thm}

Let's give some words about the method used here compared with other related works.

\begin{itemize}
\item
The setting here mostly follows that of  Olsen's \cite{O3} (see also Olsen \& Winter \cite{O2}),
so we first compare our method with those introduced by them.
The ideas to prove the first item in Theorem \ref{t1} are similar to them.
But the difference is that we are in a non-finite Markov setting
and some well known results in multifractal analysis cannot be applied directly.
For the second item, they applied the large deviation theory,
while a simple Lebesgue covering lemma is enough for us.

\item Quite recently, B. Li and J. Li \cite{LL} considered the dimension of the set $E_{=[a,b]}$ when $\psi(x)=\omega_1(x,\beta),$ where $\omega_1(x,\beta)$ is the first digit of the $\beta$-expansion of $x$. More precisely, they considered the set $$
{E}_{=[a,b]}(\beta, \omega_1):=\Bigg\{x\in [0,1): \A\bigg\{\frac{1}{n}\sum_{j=1}^n \omega_j(x,\beta)\bigg\}_{n\ge 1}=[a,b]\Bigg\}.
$$ Their method is due to Schmeling \cite{Schme}, which is one very useful method in studying $\beta$-expansions. But it is not applicable here for a general function $\psi$. To make it clear, let us cite Schmeling's idea:

Let $\beta_0<\beta$. One considers two systems $([0,1),T_{\beta_0})$ and $([0,1), T_{\beta})$. Then define a map $g$ between these two systems. More precisely, for any $x\in [0,1)$, let $$
x=\frac{\omega_1(x,\beta_0)}{\beta_0}+\frac{\omega_2(x,\beta_0)}{\beta_0^2}+\cdots+\frac{\omega_n(x,\beta_0)}{\beta_0^n}+\cdots
$$ be its $\beta_0$-expansion. Then define $$
g(x)=\frac{\omega_1(x,\beta_0)}{\beta}+\frac{\omega_2(x,\beta_0)}{\beta^2}+\cdots+\frac{\omega_n(x,\beta_0)}{\beta^n}+\cdots.
$$ It was proved by Schmeling that $g$ is $\frac{\log \beta_0}{\log \beta}$-H\"{o}lder continuous. Thus if there is a set $E'$ in $([0,1),T_{\beta_0})$ such that $g(E')$ is a subset of $E$ in $([0,1),T_{\beta})$, then it gives a lower bound of $\hdim E$.

Now we turn back to the set ${E}_{=[a,b]}(\beta,\omega_1)$. It is clear that for each $n\ge 1$, $\omega_n(g(x),\beta)=\omega_n(x,\beta_0)$, so $$
g\Big({E}_{=[a,b]}(\beta_0, \omega_1)\Big)\subset {E}_{=[a,b]}(\beta,\omega_1).
$$

This enables them  need only pay attention to the case when $\beta$ is a Parry number. However, for a general function $\psi$, there is no clear relation between $$
\sum_{j=0}^{n-1}\psi(T_{\beta_0}^j(x))\ \ \ {\text{and}}\ \ \sum_{j=0}^{n-1}\psi(T_{\beta}^j(g(x))).
$$

So we have to find other way out. The method used in this paper is a successive approximation of $\beta$ by Parry numbers, i.e. in different stage of the construction of a Moran subset of $E$, we use different  Parry number to approximate $\beta$.
\end{itemize}

\section{$\beta$-expansions}
In this section, we recall some basic properties of $\beta$-expansions.
The $\beta$-expansion was first introduced by R\'{e}nyi \cite{R}, which is given by the following algorithm. Let $\beta>1$ and define
$$T_{\beta}(x)=\beta x-\lfloor \beta x\rfloor,\ x\in[0,1).$$

By taking
$$\omega_n(x,\beta)=\lfloor\beta T_{\beta}^{n-1}x\rfloor$$
recursively for each $n\geq1,$ every $x\in[0,1)$ can be uniquely expanded into a finite or an infinite sequence
\begin{equation}\label{f3}
x=\frac{\omega_1(x,\beta)}{\beta}+\frac{\omega_2(x,\beta)}{\beta^2}+\cdots+\frac{\omega_n(x,\beta)}{\beta^n}+\cdots.
\end{equation}
Call the series (\ref{f3}) the $\beta$-expansion of $x$ and the sequence $\{\omega_n(x,\beta)\}_{n\geq1}$ the digit sequence of $x.$ We also write (\ref{f3}) as $x=\big(\omega_1(x,\beta),\cdots,\omega_n(x,\beta),\cdots\big).$

A finite or an infinite sequence $(\omega_1,\omega_2,\cdots)$ is said to be \emph{admissible} (with respect to the base $\beta$), if there exists an $x\in[0,1)$ such that the digit sequence (in the $\beta$-expansion) of $x$ begins with
$\omega_1,\omega_2,\cdots.$

Denote by $\Sigma_{\beta}^n$ the collection of all $\beta$-admissible sequences of length $n$ and by $\Sigma_{\beta}$ that of all infinite admissible sequences.

Now let us turn to the $\beta$-expansion of $1$, which plays a crucial role in
studying $\beta$-expansions. Let
$$1=\frac{\omega_1(1, \beta)}{\beta}+\cdots+\frac{\omega_n(1, \beta)}{\beta^n}+\cdots,$$ be the $\beta$-expansion of 1.
If it terminates, i.e. there exists $m\geq 1$ such that $\omega_m(1,\beta)\geq1$ but $\omega_n(1, \beta)=0$ for $n>m$ (those $\beta$ are called Parry numbers),
 we put
$(\omega_1^*(\beta),\omega_2^*(\beta),\cdots)
=(\omega_1(1, \beta),\cdots,\omega_{m-1}(1, \beta),\omega_m(1, \beta)-1)^{\infty},$
where $\omega^{\infty}$ denotes the periodic sequence $(\omega, \omega, \omega, . . .).$
Otherwise,
%%%%%%%%%%%%%%%%%%%%%%%%%%%%%%%%%%%%%%%%%%%%%%%%%%%%%%%%%%%%%%%
we put $(\omega_1^*(\beta),$ $\omega_2^*(\beta),\cdots)$ the infinite digit sequence
%%%%%%%%%%%%%%%%%%%%%%%%%%%%%%%%%%%%%%%%%%%%%%%%%%%%%%%%%%%%%%%
$(\omega_1(1,\beta),\omega_2(1,\beta),\cdots)$.
In both cases, we call
the sequence $(\omega_1^*(\beta),\omega_2^*(\beta),\cdots)$ the infinite $\beta$-expansion of unity.
% and we have always that
%$$1=\frac{\omega_1^*(\beta)}{\beta}+\cdots+\frac{\omega_n^*(\beta)}{\beta^n}+\cdots$$

The lexicographical order $<_{{\text{lex}}}$ between two infinite sequences is defined as follows
$$(\omega_1,\omega_2,\cdots,\omega_n,\cdots)<_{{\text{lex}}}(\omega_1{'},\omega_2{'},\cdots,\omega_n{'},\cdots)$$
if there exists $k\geq1$ such that $\omega_j=\omega_j{'}$ for $1 \leq j < k$, while $\omega_k<\omega_k{'}$.
%This order
%can be extended to finite blocks by identifying a finite block $(\omega_1{'},\omega_2{'},\cdots,\omega_n{'})$ with the
%sequence $(\omega_1{'},\omega_2{'},\cdots,\omega_n{'},0,0,\cdots)$.

The following result due to Parry \cite{P} is a criterion for the admissibility of a sequence.
\begin{thm}[\cite{P}]
\label{t2}
\begin{enumerate}
  \item Let $\beta>1.$ A sequence of non-negative integers $\omega=(\omega_1,\omega_2,\cdots)$ belongs to $\Sigma_{\beta}$ if and only if
  $$\forall k\geq1,\ (\omega_k,\omega_{k+1},\cdots)<_{{\text{lex}}}(\omega_1^*(\beta),\omega_{2}^*(\beta),\cdots).$$
  \item The digit sequence $(\omega_1^*(\beta),\omega_{2}^*(\beta),\cdots)$ of the $\beta$- expansion of unity is monotone as a function of $\beta.$ Therefore, if $1<\beta_1<\beta_2,$
      $$\Sigma_{\beta_1}\subset\Sigma_{\beta_2},\ \Sigma_{\beta_1}^n\subset\Sigma_{\beta_2}^n\ (\forall n\geq1).$$
\end{enumerate}
\end{thm}

%The following result of R\'{e}nyi implies that the dynamical system $([0,1], T_{\beta})$ admits $\log\beta$ as its topological entropy.
%
%\begin{thm}[\cite{R}]\label{t3}
% Let $\beta>1.$ For any $n\geq1,$
%$$\beta^n\leq\sharp\Sigma_{\beta}^n\leq \beta^{n+1}/(\beta-1),$$
%here and hereafter $\sharp$ denotes the cardinality of a finite set.
%\end{thm}

For any $\beta$-admissibe sequence $(\omega_1,\cdots,\omega_n)\in\Sigma_{\beta}^n,$ the set
$$I_{n,\beta}(\omega_1,\cdots,\omega_n):=\Big\{x\in[0,1):\omega_j(x,\beta)=\omega_j,1\leq j\leq n\Big\}$$
is called a cylinder of order $n$ (with respect to the base $\beta$). From the algorithm of $\beta$-expansion, it is clear that the length of $I_{n,\beta}(\omega_1,\cdots,\omega_n)$ satisfies that
$$|I_{n,\beta}(\omega_1,\cdots,\omega_n)|\leq \beta^{-n},$$
where $|\cdot|$ denotes the length of an interval.
If there is an equality, we call $I_{n,\beta}(\omega_1,\cdots,\omega_n)$ a \emph{full cylinder}.

The following simple fact will be referred to frequently, so we state it as a lemma. %in estimating the ergodic sum $S_ng(x):=\sum^{n-1}_{j=0}g(T^jx).$
\begin{lem}\label{l1}
Assume that $\psi$ is a continuous function. Let $\epsilon>0$ and $N\in \mathbb{N}$. Then there exists $n_0\in \mathbb{N}$ such that for any $n\ge n_0$ and $x\in [0,1)$, $$
\Big|\frac{1}{n+N}S_{n+N}\psi(x)-\frac{1}{n}S_n\psi(x)\Big|<\epsilon.
$$
\end{lem}
%$I_n$ a basic interval of order $n$. If necessary, we write $I_{n,\beta}$ to emphasize the
%dependence of the base $\beta$.
%$I_n(x)$: the basic interval of order $n$ containing $x$.

\section{Relationship between $(\Sigma_{\beta},\,T_{\beta})$ and its subsystems}
In this section, we approximate $\beta$ from below by a sequence of Parry numbers $\beta_{N}$, and then introduce a quantitative relationship between
$(\Sigma_{\beta},T_{\beta})$ and $(\Sigma_{\beta_N},T_{\beta_N}),$  which is essential in our proof of the main theorem.
%Actually, the desired result can be achieved via approximating the system by its subsystems of finite type.

\subsection{Approximating $\beta$ from below}

Recall that the $\beta$-expansion of unity is denoted by  $(\omega^*_1(\beta), \omega^*_2(\beta),\ldots)$.

For each $N\in \N$ with
$\omega^*_{N}(\beta)\ge1$,
 def\mbox{}ine $\beta_N$ to be the unique positive root of the equation
 \begin{equation}\label{f4}
1=\frac{\omega_{1}^*(\beta)}{\beta_N^{1}}+\frac{\omega_2^*(\beta)}{\beta_N^{2}}+
\cdots+\frac{\omega_{N}^*(\beta)}{\beta_N^{N}}.
\end{equation}
It is clear that  $\beta_N\le \beta$. Thus for $n\geq 1$, $\Sigma^n_{\beta_N}\subset \Sigma^n_{\beta}$. Also, we have that  $\beta_N\to \beta$ as $N\to \infty$.

 %More importantly,
% by the criterion for admissible sequence (Theorem \ref{t4}), we have, for any $(w_1,\ldots,w_n)\in \Sigma_{\beta_N}^n$ and $(w'_1,\ldots,w'_m)\in \Sigma_{\beta}^m$, that
% \begin{equation}\label{ff2}
%(w_1,\cdots,w_n, 0^N,w'_1,\cdots,w'_m)\in \Sigma_{\beta}^{n+N+m}.
%\end{equation}
%Such an assertion leads to
% By the criterion of full interval (Proposition \ref{p1}), we have
\begin{pro}[\cite{FW}]\label{p2} For any $(\omega_1,\ldots, \omega_n)\in \Sigma_{\beta_N}^n$,
$I_{n+N, \beta}(\omega_1,\ldots, \omega_n, 0^N)$ is full, or equivalently, for any $v\in \Sigma_{\beta}$, the concatenation $(\omega, 0^N, v)$ is still $\beta$-admissible.\end{pro}
%. Then for any $\beta$-admissible sequence $\eta$, $(w_1,\ldots, w_n, 0^N, \eta)$ is $\beta$-admissible.  $I_{n+N,\beta}(w_1,\cdots,w_n, 0^N)$ is full.
% So,
%\begin{equation}\label{f3}
%\frac{1}{\beta^{n+N}}\le \left|I_{n, \beta}(w_1,\ldots, w_n)\right|\leq \frac{1}{\beta^n}.
%\end{equation}

\subsection{Relationship between $\Sigma_{\beta}$ and $\Sigma_{\beta_N}$}
In this short subsection, we cite a quantitative relation between $\Sigma_{\beta}$ and $\Sigma_{\beta_N}$, which is borrowed from Tan and Wang \cite{TW}.
%is borrowed from the first two authors' former work \cite{TaW},
%, where we studied the quantitative recurrence properties of the beta-dynamical system, namely the size of the following set $$
%\Big\{x\in[0,1]: \big|T_{\beta}^nx- x\big|<e^{-(f(x)+\cdots+f(T_{\beta}^{n-1}x))}, \ {\text{for inf\mbox{}initely many}}\ n\in \N\Big\}
%$$ with $f$ a positive continuous function.
%in which the quantitative recurrence properties of the beta-dynamical systems are studied.

%\smallskip

Recall that $\beta_N$ is def\mbox{}ined in (\ref{f4}). So the infinite $\beta_N$-expansion of unity is the periodic sequence $$(\omega^*_1(\beta),\ldots, \omega^*_{N-1}(\beta),\omega^*_{N}(\beta)-1)^\infty.$$
 Now we induce a $\beta_N$-admissible sequence from a $\beta$-admissible sequence.

%\smallskip

Given a $\beta$-admissible block $\omega=(\omega_1,\cdots, \omega_n)$ with length $n$,  we obtain a
 $\beta_N$-admissible sequence $\overline{\omega}$  by changing
the blocks $(\omega_1^*(\beta),\cdots, \omega_N^*(\beta))$ in $\omega$ from the left to the right with non-overlaps
 to $(\omega_1^*(\beta),\cdots, \omega_N^*(\beta)-1)$. Denote the resulting sequence by $\overline{\omega}$. %More precisely,

%The process is operated under the monitor of the criterion for admissible sequence (Theorem \ref{pp1}).

%\medskip
% \noindent \textsc{Step 1}.\quad
%Let $\overline{k}_1$ be the smallest integer such that $w_{k}\neq  w^*_{k}(\beta)$.

 % (1). If $\overline{k}_1\leq N$, we put $k_1=\overline{k}_1$ and do nothing else;
%\smallskip

%  (2). If $\overline{k}_1> N$, we change $w_N=w^*_{N}(\beta)$ to $w^*_{N}(\beta)-1$, and put $k_1=N$.

%Then we get a sequence $\overline{w_1,\cdots,w_{k_1}}, w_{k_1+1}, \cdots, w_n$.

%\medskip
% \noindent \textsc{Step 2}.\quad
%We deal with the remaining part $(w_{k_1+1},\ldots, w_n)$ as what we have done with
%$(w_{1},\ldots, w_n)$ in \textsc{Step 1}.

%Finally, we denote the resulting sequence by $\overline{w}$. It can be readily checked by the criterion for admissibility of a sequence (Theorem \ref{t4}) that
\begin{pro}\label{p3} $\overline{\omega}\in \Sigma_{\beta_N}^n$. \end{pro}

Def\mbox{}ine the map $\pi_N:\Sigma^n_{\beta}\to \Sigma^n_{\beta_N}$ as $\pi_N(\omega)=\overline{\omega}$.
%When calculating the cardinality of the $\pi_N$-inverse of $\overline{w}$, we only need to
 % count the number of blocks $(w_1^*(\beta),\cdots, w_N^*(\beta)-1)$
 % in $\overline{w}$ %from the left to the right
 % with non-overlaps. This observation yields
\begin{pro}\label{p4}
For any $\overline{\omega}\in \Sigma^n_{\beta_N}$, $$
\sharp \pi_N^{-1}(\overline{\omega})\leq 2^{\frac{n}{N}},
$$i.e., the number of the inverse of $\overline{\omega}\in \Sigma_{\beta_N}^n$ is at most $2^{\frac{n}{N}}$.
\end{pro}

These two elementary propositions have been proved very useful in studying the dimensional theory in $\beta$-expansions. For example, by using them, it was proved that the pressure function is continuous with respect to the system \cite{TW} and that the spectrum of the level set of the classic Birkhoff averages can be achieved by an approximating method \cite{TWW}. Moreover, Propositions \ref{p5},\ref{p6} and \ref{p7} below, which are crucial in our argument, are also direct consequences of these elementary observations.

%Let   $w\in \Sigma_{\beta}^n$ and $\pi_N(w)=\overline{w}$.
%By the construction, if $w$ and $\overline{w}$ differ at some position $i$,
%they must coincide %agree
%at least at the positions $i-N+1,\cdots, i-1$.
%%So the digit sequence of $w$ and $\overline{w}$ can be depicted as$$
%%\underbrace{\bigcirc, \cdots, \bigcirc, \ast}_{\geq N},\cdots,\underbrace{\bigcirc, \cdots, \bigcirc,  \ast}_{\geq N}, \underbrace{\bigcirc, \cdots, \bigcirc}_{\ell},
%%$$ where $\bigcirc$ denotes the common digit, while $\ast$ denotes the different one.
%Thus, together with continuity of $\psi$, a simple calculation yields that
%\begin{cor}\label{c1}
%For   $\epsilon>0$,  there exists $N_0\in\mathbb{N}$ such that, for
%  $N\geq N_0$ and $n\geq N$,
%  %$\overline{w}\in \Sigma_{\beta_N}^n$, and $w\in \pi^{-1}(\overline{w})$,
%  $$
%\Big|S_n\psi(x)-S_n\psi(\overline{{x}})\Big|<n\epsilon,$$
%where $x\in I_{n}(w)$ and $\overline{x}\in I_{n}(\pi_N(w))$ for any $w\in\Sigma_{\beta}^n$.
%\end{cor}

\section{Dimensional number}

%Recall that the $\beta$-expansion of unity is denoted by $(\varepsilon_1^*(\beta),\varepsilon_2^*(\beta),\cdots).$
%For each $N\in\mathds{N}$ with $\varepsilon_N^*(\beta)\geq1,$ define $\beta_N$ to be the unique positive root of the equation
%$$1=\frac{\varepsilon_1^*(\beta)}{\beta_N}+\cdot+\frac{\varepsilon_N^*(\beta_N)}{{\beta_N}^N}.$$
%It is clear that $\beta_N< \beta$. Thus for $n\geq1$,
%$\Sigma_{\beta_N}^n\subset  \Sigma_{\beta}^n.$ Also, we have that $\beta_{N}\rightarrow\beta$ as $N\rightarrow\infty.$

In this section, we define several quantities which are closely related to the dimension of the sets in question.

From now on, we fix $\beta>1$. All the cylinders in the sequel are cylinders with respect to the base $\beta$, so we write $I_n(\omega)$ for
$I_{n,\beta}(\omega)$. Also we write $\A(x)$ for $\A(\psi,x)$.

Let $\alpha\in\mathfrak{L}_{\psi}.$ For any $\epsilon>0$ and $N\in\mathbb{N}$, define
$$\F(n,\alpha,\epsilon)=\Bigg\{\nu\in\Sigma_{\beta}^n:\Big|\frac{1}{n}S_n\psi(x)-\alpha\Big|<\epsilon,
\ {\text{for some}}\ x\in I_n(\nu)\Bigg\},$$
$$\F_N(n,\alpha,\epsilon)=\Bigg\{\nu\in\Sigma_{\beta_N}^n:\Big|\frac{1}{n}S_n\psi(x)-\alpha\Big|<\epsilon,
\ {\text{for some}}\ x\in I_n(\nu)\Bigg\}.$$

When we come to construct the desired Cantor set later, we will use $\F_N(\cdot)$ instead of $\F(\cdot)$ to avoid the barriers caused by the fact that the concatenation of two $\beta$-admissible words may not be admissible (Recall Proposition \ref{p2}). The following proposition, which is a direct consequence of Proposition \ref{p4}, says that we will not lose much if we do so.
\begin{pro}[\cite{TWW}] \label{p5}For any $\alpha\in\mathfrak{L}_{\psi},$ $$\sharp\F_N(n,\alpha, \epsilon)\le \sharp\F(n,\alpha, \epsilon)\le 2^{n/N}\sharp\F_N(n,\alpha, 2\epsilon).$$ As a consequence, the dimensional number $h_{\beta}(\alpha)$ of the set $E_{\alpha}$ can also be given as
$$h_{\beta}(\alpha)=\lim_{\epsilon\rightarrow0}\limsup_{n\rightarrow\infty}\frac{\log\sharp\F(n,\alpha,\epsilon)}{n\log\beta}=
\lim_{\epsilon\rightarrow0}\lim_{N\rightarrow\infty}\limsup_{n\rightarrow\infty}\frac{\log\sharp\F_{N}(n,\alpha,\epsilon)}{n\log\beta}.$$
\end{pro}

We restate the above proposition in another way, which will be frequently referred to.
\begin{pro}\label{p6}
For any $\delta>0,$ one can choose $\epsilon=\epsilon(\delta)>0,$ an integer $N=N(\delta, \epsilon)\in\mathbb{N}$ with $\omega_N^{*}(\beta)\geq1,$ and then $n\gg N$ correspondingly such that
 \begin{equation}\label{f5}
   \Bigg|\frac{\log\sharp\F_{N}(n,\alpha,\epsilon)}{n\log\beta}-h_{\beta}(\alpha)\Bigg|<\delta.
 \end{equation}
\end{pro}
 The concavity and continuity of $h_{\beta}(\alpha)$ were proved in Propositions 4.3 and 4.4 of \cite{TWW}.
\begin{pro}[\cite{TWW}]\label{p7}
$h_{\beta}(\alpha)$ is concave and continuous on $\mathfrak{L}_{\psi}$.
\end{pro}

\section{{Preliminary results}}
In this section, we prove some preliminary results. Recall that
$$E_{\alpha}=\Big\{x\in[0,1):\lim_{n\rightarrow\infty}\frac{1}{n}S_n\psi(x)=\alpha\Big\}$$ and
$\hdim E_{\alpha}=h_{\beta}(\alpha).$

%The following result will be frequently used to construct the connection of Hausdorff dimension among the goal decomposition sets and the known level set $E_{\alpha}$.
\begin{pro}\label{p8}
Let $\alpha\in \mathfrak{L}_{\psi}.$ Define
$E_{\ni\alpha}=\{x\in[0,1):\A(x)\ni\alpha\}.$ Then
$$\hdim E_{\alpha}=\hdim E_{\ni\alpha}.$$
\end{pro}
\begin{proof}
It's obvious that $E_{\alpha}\subset E_{\ni\alpha}$, so $$\hdim E_{\alpha}\leq\hdim E_{\ni\alpha}.$$

Now we turn to the converse inequality. This is given by showing that $$
\hdim E_{\ni\alpha}\le h_{\beta}(\alpha).
$$

Fix $\delta>0.$ By proposition \ref{p6},
there exist $\epsilon=\epsilon(\delta)$ and $N(\delta,\epsilon)\in\mathbb{N}$ such that for any $n\geq N(\delta,\epsilon),$
\begin{equation}\label{f6}
  \log\sharp \mathbb{F}(n,\alpha,\epsilon)\leq n(h_{\beta}(\alpha)+\delta)\log\beta.
\end{equation}

On the other hand, it is clear that \begin{align*}
E_{\ni\alpha}&\subset \Big\{x\in [0,1): \big|\frac{1}{n}S_n\psi(x)-\alpha\big|<\epsilon, \ {\text{i.o.}}\ n\in \mathbb{N}\Big\}\\
&=\bigcap_{N=1}^{\infty}\bigcup_{n=N}^{\infty}\bigcup_{
\nu\in\mathbb{F}(n,\alpha,\epsilon)}I_n(\nu),\end{align*} where {\em i.o.} stands for {\em infinitely often}.
Therefore, the $s$-dimensional Hausdorff measure of $E_{\ni\alpha}$ can be estimated as
\begin{align}
\nonumber \mathcal{H}^{s}(E_{\ni\alpha}) &\leq\liminf_{N\rightarrow\infty}\sum_{n=N}^{\infty}\sum_{\nu\in\mathbb{F}(n,\alpha,\epsilon)}|I_n(\nu)|^s\\
 \nonumber&\leq\liminf_{N\rightarrow\infty}\sum_{n=N}^{\infty}\sharp\mathbb{F}(n,\alpha,\epsilon)\cdot\beta^{-ns},
\end{align}
which is finite for any $s>h_{\beta}(\alpha)+\delta$ by (\ref{f6}). The arbitrariness of $s$ yields that
$$\hdim E_{\ni\alpha}\leq h_{\beta}(\alpha)+\delta.$$
\end{proof}
%Recall that when $\psi$ is a  continuous function, $\mathfrak{L}_{\psi}$ is the set of accumulation
%points of the ergodic averages of $\psi,$ which is a closed interval, and $A(x)$ the set of accumulation
%points of the Birkhoff averages sequence of a fixed $x\in[0,1]$, that is
The following is an elementary result, which shows that $\A(x)$ is a closed interval.
%The following results shows that $\A(x)$ is still a closed interval.
\begin{pro}\label{p9}
$\A(x)$ is a closed interval, more precisely,
\begin{equation}\label{f7}\A(x)=\Bigg[\liminf_{n\rightarrow\infty}\frac{S_n\psi(x)}{n},\limsup_{n\rightarrow\infty}\frac{S_n\psi(x)}{n}\Bigg].\end{equation}
\end{pro}
\begin{proof}
It suffices to show the inclusion $``\supset"$ in (\ref{f7}). Fix a real number  $t$ with
$$\liminf_{n\rightarrow\infty}\frac{S_n\psi(x)}{n}<t<\limsup_{n\rightarrow\infty}\frac{S_n\psi(x)}{n}.$$
Then one has $$
\frac{1}{n}S_n\psi(x)<t, {\text{i.o.}}\ n\in \mathbb{N}, \ \ {\text{and}} \ \ \ \frac{1}{n}S_n\psi(x)>t, \  {\text{i.o.}}\ n\in \mathbb{N}.
$$
 This enables us to find a sequence $\{n_k\}_{k\ge 1}$ such that $$
\left(\frac{1}{n_k}S_{n_k}\psi(x)-t\right)\left(\frac{1}{n_k+1}S_{n_k+1}\psi(x)-t\right)\le 0.
$$ Thus, $$
\Big|\frac{1}{n_k}S_{n_k}\psi(x)-t\Big|\le \Big|\frac{1}{n_k}S_{n_k}\psi(x)-\frac{1}{n_k+1}S_{n_k+1}\psi(x)\Big|.
$$ By the boundedness of $\psi$, the right term turns to 0 as $k\to \infty$. This gives that $$
\lim_{k\to \infty}\frac{1}{n_k}S_{n_k}\psi(x)=t.
$$
\end{proof}

The following dimension result about homogeneous Cantor set is a classic tool to estimate the Hausdorff dimension of a fractal set from below.

Let $\{m_i\}_{{i\in\mathbb{N}}}$ be a sequence of positive integers and $\{c_i\}_{i\in\mathbb{N}}$ be a sequence of positive numbers satisfying $m_i\geq2,$ $0<c_i< 1,$
$m_1c_1\leq \delta$ and $m_ic_i\leq1~(i\geq2),$ where $\delta$ is some positive number.
Let
$$D=\bigcup_{i\geq0}D_i,~~~~ D_0 =\emptyset, ~~D_i =\Big\{(\sigma_1,\cdots,\sigma_i): 1\leq \sigma_j\leq m_j,1\leq j\leq i \Big\}.$$
If $\sigma = (\sigma_1,\cdots,\sigma_k)\in D_k, \tau = (\tau_1,\cdots,\tau_m) \in D_m,$ the concatenation of $\sigma$ and $\tau$ is denoted by
$$\sigma\ast\tau=(\sigma_1,\cdots,\sigma_k,\tau_1,\cdots,\tau_m).$$

\begin{dfn}[\cite{FWW}]\label{d1}
Let $(X,d)$ be a metric space. Suppose that $J\subset X$ is a closed subset with diameter $\delta>0.$ Let $\mathfrak{F}=\{J_{\sigma}:\sigma\in D\}$ be a collection of
closed subsets of $J$ with the properties

\noindent(1) $J_{\emptyset}= J$;

\noindent(2) For any $i\geq1$ and $\sigma\in D_{i-1},$ $J_{\sigma\ast 1}, J_{\sigma\ast 2},\cdots, J_{\sigma\ast {m_i}}$ are subsets of $J_{\sigma}$ and $int( J_{\sigma\ast i})\bigcap int( J_{\sigma \ast j}) = \emptyset \ (i\neq j),$ where $int(\cdot) $ denotes the interior of a set;

\noindent(3) For any $i\geq 1$ and $\sigma\in D_{i-1}, 1\leq \ell \leq m_i,$
$$\frac{| J_{\sigma\ast \ell}|}{| J_{\sigma} | }= c_i,$$
where $|\cdot|$ denotes the diameter.
Then
$$\mathds{C}_{\infty}:=\bigcap_{i\geq1}\bigcup_{\sigma\in D_{i}}J_{\sigma}$$
 is called a homogeneous Cantor set determined by $\mathfrak{F}$. For each $i\ge 1$, we call the union $\mathds{C}_i:=\bigcup_{\sigma\in D_{i}}J_{\sigma}$ the $i$th generation of $\mathds{C}_{\infty}$.
\end{dfn}

\begin{lem}[\cite{FWW}]\label{l2}
For the homogeneous Cantor set defined above, we have
$$\hdim \mathds{C}_{\infty}\geq \liminf_{i\to \infty}\frac{\log m_1\cdots m_i}{-\log c_1\cdots c_{i+1}m_{i+1}}.$$
\end{lem}

\section{Proof of Theorem \ref{t1}}
Recall that $\mathfrak{L}_{\psi}$ is a closed interval. So without lost of generality, we assume that $[a,b]\subset\mathfrak{L}_{\psi}.$
%In fact, the definition of $E_{=[a,b]}$ and $E_{\supset[a,b]}$ contain that $a,b\in\mathfrak{L}_{\psi}$ which is just the desired case.
%As for the other case $E_{\subset[a,b]}$, when $[a,b]\bigcap\mathfrak{L}_{\psi}$ is a non-degenerate proper subset of $[a,b],$
%write$[a_1,b_1].$ Then the results holds for $a_1,b_1.$ Since the points in $[a,b]$ but out of $[a_1,b_1]$ make no difference when
%we compute the value of both side of the third equation in Theorem \ref{t1}, the results still hold for $a,b.$ In the following proof,
%we always assume that $a,b\in\mathfrak{L}_{\psi}.$

\subsection{The first item in Theorem \ref{t1}}\

On the one hand, by the concavity of $h_{\beta}(\cdot)$, one has $$
\min\Big\{h_{\beta}(\alpha): \alpha\in [a,b]\Big\}=\min\Big\{h_{\beta}(a), h_{\beta}(b)\Big\}:=s_*.
$$

On the other hand, it is obvious that $E_{=[a,b]}\subset E_{\supset[a,b]}$. So to get the desired result, it suffices
to show that $$s_{*}\leq \hdim E_{=[a,b]}, \ \ \ \hdim E_{\supset[a,b]}\leq s_{*}.$$
%which is given in the next two propositions.

The second inequality is a direct corollary from Proposition \ref{p8}, so we need only pay attention to the first one.

Since  $\A(x)$ is a closed interval (Proposition \ref{p9}), we can rewrite $E_{=[a,b]}$ as
$$E_{=[a,b]}=\Bigg\{x\in[0,1):\liminf_{n\rightarrow\infty}\frac{1}{n}S_n\psi(x)=a,\limsup_{n\rightarrow\infty}\frac{1}{n}S_n\psi(x)=b\Bigg\}.$$

We will construct a homogeneous Cantor set $\mathds{C}_{\infty}\subset E_{=[a,b]}$ with Hausdorff dimension bounded from below by $s_{*}.$

We fix some notations. \begin{itemize}\item At first, fix $\delta>0.$ By Proposition \ref{p6}, we can choose a sequence of triples $(\epsilon_k, N_k, n_k)_{k\ge 1}$ (recursively) such that $\epsilon_k \to 0$ as $k\to \infty$, $N_k>N_{k-1}$, $n_k/(n_k+N_k)\ge 1-\delta$ and for each $k\ge 1$,
\begin{equation}\label{f8}
  \Bigg|\frac{\log\sharp\F_{N_{2k-1}}(n_{2k-1},a,\epsilon_{2k-1})}{n_{2k-1}\log\beta}-h_{\beta}(a)\Bigg|<\delta
\end{equation}
\begin{equation}\label{f9}
  \Bigg|\frac{\log\sharp\F_{N_{2k}}(n_{2k},b,\epsilon_{2k})}{n_{2k}\log\beta}-h_{\beta}(b)\Bigg|<\delta.
\end{equation}
%In fact, by Proposition \ref{p2}, for any $\delta>0,$ choose $\epsilon_{2k-1}<\frac{1}{2k-1}$  and integers $N_{2k-1}\in\mathds{N},$ $n_{2k-1}\gg N_{2k-1},$ according to ensure that the Formula (\ref{e2}) holds, i.e. (\ref{e5}) holds.
%Likewise, choose $\epsilon_{2k}<\frac{1}{2k}$  and integers $N_{2k}>N_{2k-1},$ $n_{2k}\gg N_{2k},$ according to ensure (\ref{e6}) holds.
%
%Let $\{l_k\}_{k\geq0}$ be a strictly increasing integer sequence satisfying $l_k\gg n_{k+1}+N_{k+1}$ and $l_0=0.$
\item Secondly, we choose a sequence of integers $\{\ell_k\}_{k\ge 1}$ such that for each $k\ge 1$,
 \begin{equation}\label{e1}
\ell_k\gg n_{k+1}+N_{k+1}, \ \ \ell_{k}\gg \ell_1(n_1+N_1)+\cdots+\ell_{k-1}(n_{k-1}+N_{k-1})\end{equation}
and let $n_0=N_0=l_0=0.$
\item Thirdly, we define
\begin{eqnarray*} \mathbb{D}_{k}=\left\{
                                  \begin{array}{ll}
                                    \Big\{\omega=(v,0^{N_k}):v\in\F_{N_{k}}(n_{k},a,\epsilon_{k})\Big\}, & \hbox{when $k$ is odd;} \bigskip\\
                                    \Big\{\omega=(v,0^{N_k}):v\in\F_{N_{k}}(n_{k},b,\epsilon_{k})\Big\}, & \hbox{when $k$ is even.}
                                  \end{array}
                                \right.
\end{eqnarray*}
\item At last, we give an estimation. If $n_k$ is chosen sufficiently large compared with $N_k$, then by Lemma \ref{l1}, for each $\omega\in \mathbb{D}_k$ and every $x\in I_{n_k+N_k}(\omega)$, we have
\begin{equation}\label{e2}
  \Bigg|\frac{1}{n_k+N_k}\sum_{j=0}^{n_k+N_k-1}\psi(T_{\beta}^jx)-a \ (\text{or} \ b)\Bigg|<2\epsilon_k,
\end{equation}
according as $k$ is odd or even.
\end{itemize}
%$$\mathbb{D}_{l}^2=\Big\{(\omega,0^{N_l})\in\Sigma_{\beta_{N_l}}^{n_l+N_l}:\omega\in\F_{N_{l}}(n_{l},b,\epsilon_{l})\Big\}.$$
\begin{rem}\label{r1}Note that the words $v$ in $\F_{N_{k}}(n_{k},*,\epsilon_{k})$ are $\beta_{N_k}$-admissible. So the extra term $0^{N_k}$ ensures that every word $\omega$ in $\mathbb{D}_k$ can concatenate any other $\beta$-admissible words freely by Proposition \ref{p2}.
\end{rem}

Now we are in the position to construct a subset $\mathds{C}_{\infty}$ of $E_{=[a,b]}$ generation by generation. Firstly, let $\mathds{C}_0=[0,1].$
\medskip

\emph{The generations $\{\mathds{C}_i\}_{1\le i\le \ell_1}$.}
For each $1\le i\le \ell_1$, set
$$\mathds{C}_i=\bigcup_{\omega_1\in \mathbb{D}_1, \cdots, \omega_i\in \mathbb{D}_1}I_{i(n_1+N_1)}(\omega_1,\cdots, \omega_i).$$
By the remark given above and Proposition \ref{p2}, all the cylinders in $\mathds{C}_i$ are full cylinders of order $i(n_1+N_1)$. Then the ratio of the diameter of a cylinder in $\mathds{C}_{i-1}$ with that of a cylinder in $\mathds{C}_i$ is $\beta^{-(n_1+N_1)}$. Moreover, it is also clear that each cylinder in $\mathds{C}_{i-1}$ contains  $\sharp\mathbb{D}_1$ elements in $\mathds{C}_i$. So, by borrowing the notaton from Definition \ref{d1} of a homogeneous Cantor set, we have
\begin{equation}\label{e3}
c_i=\beta^{-(n_1+N_1)},\ \ m_i=\sharp \mathbb{D}_1=\sharp \F_{N_1}(n_1, a, \epsilon_1), \ {\text{for}}\ 1\le i\le \ell_1.
\end{equation}

\emph{The inductive step.} Assume that the first $(\ell_1+\cdots+\ell_{k-1})$th generations have been well defined. Note that $\mathds{C}_{\ell_1+\cdots+\ell_{k-1}}$ consists of a collection of full cylinders of order $$\ell_1(n_1+N_1)+\cdots+\ell_{k-1}(n_{k-1}+N_{k-1}):=t_{k-1}.$$ For notational simplication, we write a general element in $\mathds{C}_{\ell_1+\cdots+\ell_{k-1}}$ as $I_{t_{k-1}}(\omega^{(k-1)})$.
Note that here $\omega^{(k-1)}$ is a word of length $t_{k-1}$ which can concatenate any other $\beta$-admissible words(by Remark \ref{r1}).

Now we define the generations $$
\Big\{\mathds{C}_i: \ell_1+\cdots+\ell_{k-1}<i\le \ell_1+\cdots+\ell_{k-1}+\ell_k\Big\}.
$$
Write $j=i-(\ell_1+\cdots+\ell_{k-1})$. Then set $$
\mathds{C}_i=\bigcup_{I_{t_{k-1}}(\omega^{(k-1)})\subset \mathds{C}_{\ell_1+\cdots+\ell_{k-1}}}\bigcup_{\omega_1\in \mathbb{D}_k,\cdots,\omega_j\in \mathbb{D}_k}I_{t_{k-1}+i(n_k+N_k)}(\omega^{(k-1)}, \omega_1,\cdots,\omega_j).
$$ Similar to (\ref{e3}), we have, for $\ell_1+\cdots+\ell_{k-1}<i\le \ell_1+\cdots+\ell_{k-1}+\ell_k$, \begin{equation}\label{e4}
c_i=\beta^{-(n_k+N_k)},\ \ m_i=\sharp \mathbb{D}_k=\sharp \F_{N_k}(n_k, a, \epsilon_k) \ \ {\text{or}}\ \ \sharp \F_{N_k}(n_k, b, \epsilon_k)
\end{equation}  according as $k$ is odd or even.

The desired homogeneous Cantor set is defined as
 $$\mathds{C}_{\infty}=\bigcap_{i=0}^{\infty}\mathds{C}_i.$$
We claim that
\begin{pro}\label{p10}
 $\mathds{C}_{\infty}\subset E_{=[a,b]}$.
\end{pro}

\begin{proof}
The proof is done by some elementary estimations, so we only check that for each $x\in \mathds{C}_{\infty}$,\begin{equation}\label{e5}
\limsup_{n\to \infty}\frac{1}{n}S_n\psi(x)\le b\ \ {\text{and}}\ \ \lim_{k\to\infty}\frac{1}{t_{2k}}S_{t_{2k}}\psi(x)\ge b.
\end{equation}

Bear in mind the construction of $\mathds{C}_{\infty}$, the formula (\ref{e2}) and $$
t_k=\ell_1(n_1+N_1)+\cdots+\ell_k(n_k+N_k).
$$

(i). For each $n\gg 1$, let $k$ and then $0\le \ell< \ell_k$, $0\le j< n_k+N_k$ be the integers such that $$
t_{k-1}\le n=t_{k-1}+\ell (n_k+N_k)+j<t_k.
$$
Since $a\le b$, then by (\ref{e2}), we have \begin{align*}
S_n\psi(x)\le \ell_1\cdot (n_1+N_1)(b+2\epsilon_1)&+\cdots+ \ell_{k-1}\cdot (n_{k-1}+N_{k-1})(b+2\epsilon_{k-1})\\
&+\ell \cdot (n_k+N_k) (b+2\epsilon_k)+j \|\psi\|_{\infty}.
\end{align*} Recall the choice of $\ell_k$ (see the first formula in (\ref{e1})) and the fact that $\epsilon_k\to 0$ as $k\to \infty$. One can say that $$
n_i+N_i=n_i+o(n_i), \ \  j\|\psi\|_{\infty}=o(\ell_{k-1})=o(t_{k-1}).
$$ Thus it follows that $$
S_n\psi(x)\le b\Big(t_{k-1}+\ell (n_k+N_k)\Big)+o(t_{k-1})\le b\cdot n+o(n).
$$ Therefore,
the first assertion in (\ref{e5}) follows.

(ii). By the second formula in (\ref{e1}) and the inequality (\ref{e2}), we have $$
S_{t_{2k}}\psi(x)\ge o(\ell_{2k})+\ell_{2k}(n_{2k}+N_{2k})(b-2\epsilon_{2k}).
$$ Using the second formula in (\ref{e1}) again, which enables us to say that $\ell_{2k}(n_{2k}+N_{2k})=t_{2k}+o(t_{2k})$, the second assertion in (\ref{e5})  thus follows.
%\hfill $\Box$
%$$q_{2m-1}=\sum_{j=1}^{2m-1}l_j(n_j+N_j)\ \ \text{and}\  \ q_{2m}=\sum_{j=1}^{2m}l_j(n_j+N_j).$$
%For $x\in \mathds{C}_{\infty},$ By the construction of the set $\mathds{C}_{\infty}$ and the estimation of some ergodic average (\ref{e8}) and (\ref{e9}), we have
%%$\Bigg|\frac{1}{q_{2m-1}}\sum_{j=0}^{q_{2m-1}-1}f(T^jx)-b\Bigg|<2\epsilon_1.$
%$$\frac{S_{q_{2m-1}}\psi(x)}{q_{2m-1}}\geq
%\frac{\sum_{i=1}^{m-1}l_{2i}(n_{2i}+N_{2i})(b+2\epsilon_{2i})+\sum_{i=1}^{m}l_{2i-1}(n_{2i-1}+N_{2i-1})(a+2\epsilon_{2i-1})}{q_{2m-1}}$$
%and
%$$\frac{S_{q_{2m-1}}\psi(x)}{q_{2m-1}}\leq
%\frac{\sum_{i=1}^{m-1}l_{2i}(n_{2i}+N_{2i})(b-2\epsilon_{2i})+\sum_{i=1}^{m}l_{2i-1}(n_{2i-1}+N_{2i-1})(a-2\epsilon_{2i-1})}{q_{2m-1}}.$$
%By a simple computation with $l_{k}\gg n_{k+1}+N_{k+1},$ we have
%$$\lim_{m\rightarrow\infty}\frac{S_{q_{2m-1}}\psi(x)}{q_{2m-1}}=a.$$
%Similarly, we can prove that
%$$\lim_{m\rightarrow\infty}\frac{S_{q_{2m}}\psi(x)}{q_{2m}}=b.$$
%From the structure of the set $\mathds{C}_{\infty},$ we have
%$$\liminf_{n\rightarrow\infty}\frac{S_n\psi(x)}{n}=a,\ \limsup_{n\rightarrow\infty}\frac{S_n\psi(x)}{n}=b.$$
%Hence $\mathds{C}_{\infty}\subset E_{=[a,b]}.$
\end{proof}

\subsection{Hausdorff dimension of $\mathds{C}_{\infty}$.} We apply Lemma \ref{l2} to give the lower bound of $\hdim \mathds{C}_{\infty}$.

Now we collect the information about $c_i$ and $m_i$: for $\ell_1+\cdots+\ell_{k-1}<i\le \ell_1+\cdots+\ell_{k-1}+\ell_k$, \begin{align*}
c_i=&\beta^{-(n_k+N_k)},\\ \ \beta^{n_k(h_{\beta}(a)-\delta)}\le m_i\le \beta^{n_k(h_{\beta}(a)+\delta)}
\ \ & {\text{or}}\ \ \ \ \beta^{n_k(h_{\beta}(b)-\delta)}\le m_i\le \beta^{n_k(h_{\beta}(b)+\delta)},
\end{align*} according as $k$ is odd or even (see the formulae (\ref{e4}),(\ref{f8}) and (\ref{f9})).

Since $\ell_{k-1}\gg n_k+N_k$, the term $\log (c_{i+1}\cdot m_{i+1})$ is negligible compared with $\log (c_1\cdots c_i)$.  So, by Lemma \ref{l2}, \begin{align*}
\hdim \mathds{C}_{\infty}&\ge \liminf_{i\to \infty}\frac{\log (m_1\cdots m_i)}{-\log (c_1\cdots c_i)}\ge \inf_{i\ge 1}\frac{\log m_i}{-\log c_i}\\&\ge (1-\delta)\min\Big\{h_{\beta}(a)-\delta, h_{\beta}(b)-\delta\Big\}.
\end{align*}
Then by the arbitrariness of $\delta>0$, it follows that $$
\hdim E_{=[a,b]}\ge \min\Big\{h_{\beta}(a), h_{\beta}(b)\Big\}.
$$
%Now we are ready to prove that $$\hdim \mathds{C}_{\infty}\geq \min\{h_{\beta}(a),h_{\beta}(b)\}.$$
%Actually $\mathds{C}_{\infty}$ is a homogeneous Moran set associated with the number of $k$-th intervals, and we can get the corresponding number $n_k,c_k$ in Lemma \ref{l2}, write $s_k=n_k$ to avoid the confusion of denotes.
%
%(i)$\sum_{i=0}^{2m-2}l_i<k\leq \sum_{i=0}^{2m-1}l_i$ for some $m\in\mathds{N},$ then $s_k=\sharp\F_{N_{2m-1}}(n_{2m-1},a,\epsilon_{2m-1})$ and $c_k=\beta^{-(n_{2m-1}+N_{2m-1})};$
%
%(ii)$\sum_{i=0}^{2m-1}l_i<k\leq \sum_{i=0}^{2m}l_i$ for some $m\in\mathds{N},$ then $s_k=\sharp\F_{N_{2m}}(n_{2m},b,\epsilon_{2m})$ and $c_k=\beta^{-(n_{2m}+N_{2m})}.$
%%According to the construction of $\mathbb{D}_{\infty},$
%
%For any $k\geq l_1,$ there exist integers $n\geq1$ and $0\leq p<l_{k+1}$ such that
%$$\sum_{i=1}^ml_i(n_i+N_i)+p(n_{m+1}+N_{m+1})\leq n<\sum_{i=1}^ml_i(n_i+N_i)+(p+1)(n_{m+1}+N_{m+1}).$$
%Let $\eta<\min\{h_{\beta}(a),h_{\beta}(b)\}-\delta.$ By lemma \ref{l2}, (\ref{e5}) and (\ref{e6}),
%the lower bound of $\hdim \mathds{C}_{\infty}$ is given by
%\begin{align*}
% \liminf_{k\rightarrow\infty}\frac{\sum_{j=1}^{m}{l_j}n_j\eta\log\beta+pn_{m+1}\eta\log\beta}
%{\sum_{j=1}^{m}{l_j(n_j+N_j)}\log\beta+(p+1)(n_{m+1}+N_{m+1})\log\beta-{n_{m+1}}\eta\log\beta}=\eta.
%\end{align*}
%For the arbitrariness of $\eta$ and $\delta,$
%$$\hdim \mathds{C}_{\infty}\geq \min\{h_{\beta}(a),h_{\beta}(b)\}.$$

\subsection{The second item in Theorem \ref{t1}}

Recall that $$E_{\subset[a,b]}=\big\{x:\A(x)\subset[a,b]\big\}.$$ Write $s^{*}=\sup\big\{h_{\beta}(\alpha):\alpha\in[a,b]\big\}$.
%The proof is divided into two parts. The lower bound can be obtained by a obvious relationship $E_{\alpha}\subset E_{\subset[a,b]}.$

\textbf{Lower bound.} It is clear  that for each $\alpha\in [a,b]$, $E_{\alpha}\subset E_{\subset [a,b]}$. Thus
%The lower bound can obtained by the following fact that for any $\alpha\in[a,b]$
$$\hdim E_{\subset[a,b]}\ge h_{\beta}(\alpha), \ {\text{for all}}\ \alpha\in [a,b].$$
%In fact, for any $\alpha\in[a,b]$ and $x\in E_{\alpha},$ we have
%$A(x)=\{\alpha\}\subset[a,b].$ Then $x\in E_{\subset[a,b]}.$
%Hence $E_{\alpha}\subset E_{\subset[a,b]}.$ Thus
%$$s^{*}\leq \hdim E_{\subset[a,b]}.$$

\textbf{Upper bound.} To get the upper bound, we first cite a Lebesgue covering Lemma from\cite{W}:
\begin{lem}[\cite{W}]\label{l3}
 Let $(X, d)$ be a compact metric space and $\mathfrak{O}$  an open cover of $X$. Then there exists a $\delta>0$ such that each subset of $X$ of diameter less than or equal to $\delta$ lies in some member of $\mathfrak{O}.$ (Such a $\delta$ is called a Lebegue number of $\mathfrak{O}$.)
\end{lem}
Fix $\eta>0$. Recall the definition of $h_{\beta}(\alpha)$ (Proposition \ref{p5}). Then for any $\alpha\in[a,b],$ there exist $\epsilon_{\alpha}=\epsilon(\alpha,\eta)$ and $N(\eta,\epsilon_{\alpha})\in\mathbb{N}$ such that for any $n\geq N(\eta,\epsilon_{\alpha}),$ we get
$$ \log\sharp \mathbb{F}(n,\alpha,\epsilon_{\alpha})\leq n(h_{\beta}(\alpha)+\eta)\log\beta.$$
So for each $\alpha\in[a,b],$ we have
\begin{equation}\label{e6}
  \log\sharp \mathbb{F}(n,\alpha,\epsilon_{\alpha}) \leq n(s^{*}+\eta)\log\beta.
\end{equation}

Clearly, $\big\{B(\alpha,\epsilon_{\alpha}):\alpha\in[a,b]\big\}$ is an open covering of the closed interval $[a,b].$ Then by Lemma \ref{l3}, let $\delta$ be a lebesgue number of this cover. Take $\big\{[a_k,b_k]:k\geq1\big\}$ a collection of closed intervals with diameter at most $\frac{\delta}{4}$ such that $$[a,b]=\bigcup_{k\geq1}[a_k,b_k].$$ Then
$$E_{\subset[a,b]}=\Big\{x:\A(x)\subset[a,b]\Big\}\subset
\bigcup_{k\geq1}E^{(k)}$$
where $E^{(k)}=\big\{x:\A(x)\cap[a_k,b_k]\neq\emptyset\big\}.$

For each $k\ge 1$, by using Lebesuge covering Lemma, there exists $\alpha\in [a,b]$ such that $$
[a_k-\delta/4, b_k+\delta/4]\subset B(\alpha, \epsilon_{\alpha}).
$$

On the other hand, for each $x\in E^{(k)},$ there exists $t_x\in [a_k,b_k]$ such that $t_x\in \A(x).$ Then there are infinitely many $n$ such that
$$\Big|\frac{1}{n}S_n\psi(x)-t_x\Big|<{\delta}/{4}.$$
Since
$$B\Big(t_x,{\delta}/{4}\Big)\subset\Big[a_k-{\delta}/{4},b_k+{\delta}/{4}\Big]\subset B(\alpha,\epsilon_{\alpha}),$$ we have
$$ E^{(k)}\subset\Bigg\{x:\Big|\frac{1}{n}S_n\psi(x)-\alpha\Big|<\epsilon_{\alpha}\ {\text{i.o.}}\ n\in \N\Bigg\}.$$

Therefore,
$$E^{(k)}\subset\bigcap_{N=1}^{\infty}\bigcup_{n=N}^{\infty}\bigcup_{
\nu\in\mathbb{F}(n,\alpha,\epsilon_{\alpha})}I_n(\nu).$$
Then the $s$-dimensional Hausdorff measure of $E^{(k)}$ can be estimated as
\begin{align}
\nonumber \mathcal{H}^{s}(E^{(k)})
&\leq\liminf_{N\rightarrow\infty}\sum_{n=N}^{\infty}\sum_{\nu\in\mathbb{F}(n,\alpha,\epsilon_{\alpha})}|I_n(\nu)|^s\\
 \nonumber&\leq\liminf_{N\rightarrow\infty}\sum_{n=N}^{\infty}\sharp\mathbb{F}(n,\alpha,\epsilon_{\alpha})\cdot\beta^{-ns},
\end{align}
which is finite for any $s>s^{*}+\eta$ by (\ref{e6}). The arbitrariness of $s$ yields that
$$\hdim E^{(k)}\leq s^{*}+\eta.$$
Hence
$$\hdim E_{\subset[a,b]}\leq \sup_{k\geq1}\hdim E^{(k)}\leq s^{*}+\eta.$$



%\begin{acknowledgements}
%\end{acknowledgements}



\end{document}